\documentclass[12pt,leqno,oneside]{amsart}

\usepackage{amsmath,amsthm,amssymb,amsfonts,ifpdf,svn}

\usepackage{enumitem}
\setlist{leftmargin=*}
\setlist[1]{labelindent=\parindent} % Only the level 1

\ifpdf
\usepackage[pdftex,pdfstartview=FitH,pdfborderstyle={/S/B/W 1},%
colorlinks=true, linkcolor=blue, urlcolor=blue, citecolor=blue,%
pagebackref=true]{hyperref}
\else
\usepackage[hypertex]{hyperref}
\fi

\usepackage[alphabetic,abbrev,msc-links,nobysame]{amsrefs}

\newtheorem {Theorem}   {Theorem}
\newtheorem {thm}    [Theorem]{Theorem}

\newtheorem {lem}      [Theorem]    {Lemma}

\newtheorem {prp}[Theorem]  {Proposition}

\newcounter{AbcT}

\newtheorem {AbcTheorem} [AbcT]{Theorem}
\newtheorem {AbcLemma} [AbcT]{Lemma}

\theoremstyle{definition}

\renewcommand{\a}{\alpha}
\renewcommand{\b}{\beta}
\renewcommand{\d}{\delta}

\newcommand{\e}{\varepsilon}
\newcommand{\f}{\varphi}

\renewcommand{\l}{\lambda}

\newcommand{\s}{\sigma}
\renewcommand{\t}{\theta}

\newcommand{\R}{{\mathbb R}}

\newcommand{\Z}{{\mathbb Z}}
\newcommand{\C}{{\mathbb C}}

\newcommand{\F}{{\mathbb F}}
\renewcommand{\P}{{\mathbb P}}

\newcommand {\cL} {{\mathcal L}}
\newcommand {\cA} {{\mathcal A}}

\newcommand{\wh}{\widehat}

\newcommand{\be}{\begin{equation}}
\newcommand{\ee}{\end{equation}}
\newcommand{\bea}{\begin{eqnarray}}
\newcommand{\eea}{\end{eqnarray}}
\newcommand{\bean}{\begin{eqnarray*}}
\newcommand{\eean}{\end{eqnarray*}}

\newcommand {\equ}[1]     {\eqref{#1}}

\newcommand {\IGNORE}[1]  {}

\newcommand {\absolute}[1] {\left| {#1} \right|}

\newcommand {\norm}[1] {\left\| {#1} \right\|}

\newcommand{\fourIdx}[5]{%
\setbox1=\hbox{\ensuremath{^{#1}}}%
\setbox2=\hbox{\ensuremath{_{#2}}}%
\setbox5=\hbox{\ensuremath{#5}}%
\hspace{\ifnum\wd1>\wd2\wd1\else\wd2\fi}%
\ensuremath{\copy5^{\hspace{-\wd1}\hspace{-\wd5}#1\hspace{\wd5}#3}%
_{\hspace{-\wd2}\hspace{-\wd5}#2\hspace{\wd5}#4}%
}}

\newcommand{\ag}{\F_p^d\rtimes \SL_d(\F_p)}
\newcommand{\CF}{\C[\F_p^d]}
\newcommand{\CFh}{\C[\wh\F_p^d]}
\newcommand{\Lq}{{L^q}}
\newcommand{\Lhq}{{\wh L^q}}
\newcommand{\Ltwo}{{L^2}}
\newcommand{\Lhtwo}{{\wh L^2}}
\newcommand{\Lfour}{{L^4}}
\newcommand{\Lhfour}{{\wh L^4}}

\DeclareMathOperator{\SL}{SL}

\DeclareMathOperator{\SO}{SO}

\DeclareMathOperator{\sign}{sign}
\DeclareMathOperator{\Tr}{Tr}

\title[Spectral gap for affine transformations]%
{Spectral gap in the group of affine transformations over prime fields}
\author{Elon Lindenstrauss}
\address[Elon Lindenstrauss]{The Einstein Institute of Mathematics, Edmond J. Safra Campus,
Givat Ram, The Hebrew University of Jerusalem, Jerusalem, 91904, Israel}
\email{elon@math.huji.ac.il}

\author{P\'eter P. Varj\'u}
\address[P\'eter P. Varj\'u]{University of Cambridge, DPMMS, Wilberforce Road, Cambridge, CB3 0WA, UK}
\email{pv270@dpmms.cam.ac.uk}

\thanks{
EL was supported by  the European Research Council
(Advanced Research Grant 267259) and the ISF (grant 983/09)
PV was supported by the Simons Foundation and the European Research Council
(Advanced Research Grant 267259).
The authors would like to thank the Israeli Institute for Advanced Study for its hospitality during the fall of 2013. 
}

\SVN $Date$
\date{\SVNDate}


\begin{document}

\begin{abstract}
We study random walks on the groups $\F_p^d\rtimes \SL_d(\F_p)$.
We estimate the spectral gap in terms of the spectral gap of
the projection to the linear part $\SL_d(\F_p)$.
This problem is motivated by an analogue in the group $\R^d\rtimes \SO(d)$, which
have application to smoothness of self-similar measures.
\end{abstract}

\maketitle

\section{Introduction}
Let $G$ be a finite group.
Fix a set $S\subset G$, and let $X_1,X_2\ldots\in S$
be a sequence of independent random elements taking each element of $S$ with equal probability.
Denote the product of the first $l$ by $Y_l=X_l\cdots X_1$.
The sequence $Y_1,\cdots,Y_l$ is called the (simple) random walk on $G$ generated by $S$.

We consider the following operator acting on the space $\C[G]$ of complex valued functions on $G$:
\[
\cL(G,S)f(g)=\frac{1}{|S|}\sum_{s\in S}f(s^{-1}g),
\]
for $f\in\C[G]$ and $g\in G$.
In addition, we consider $\cL_0(G,S)$, the restriction of $\cL(G,S)$
to the one codimensional subspace of $\C[G]$ consisting of functions orthogonal to the constants.
This averaging operator is intimately connected with the random walk $Y_l$, and in particular the norm of $\cL_0(G,S)$ is closely connected with how quickly this random walk becomes equidistributed. Clearly $\|\cL_0(G,S)\| \leq \|\cL(G,S)\|=1$. We shall call the difference $1-\|\cL_0(G,S)\|$ the \emph{spectral gap of the random walk}. If $S$ is symmetric the spectral gap coincides with difference between the trivial eigenvalue $1$ of $\cL(G,S)$ and the greatest eigenvalue of the operator $\|\cL_0(G,S)\|$, though in general (despite the name which seems to be fairly standard) what we call the spectral gap has no direct spectral interpretation.
The operator norm here and everywhere below is with respect to the $L^2$ norm
on the space the operator is acting on, in this case the finite group $G$ equipped with the counting measure.

It is easily seen that the random walk mixes rapidly if the spectral gap is large.
Indeed, denote by $\d_1\in\C[G]$, the function given by
$\d_1(1)=1$ and $\d_1(g)=0$ for $g\neq 1$, where $1$ denotes the multiplicative unit in $G$
and in $\C$ and in any multiplicative group.
Then one can show by induction, that for all integer $l\ge0$, the probability that $Y_l=g$ is
$\cL(G, S)^l\d_1(g)$.
We can write
\[
\d_1(g)=\frac{1}{|G|}+f(g),
\]
where $f$ is a function orthogonal to the constants.
Then
\[
\left\|\frac{1}{|G|}-\cL(S)^l\d_1\right\|_{L^\infty}\le\left\|\frac{1}{|G|}-\cL(S)^l\d_1\right\|_\Ltwo
\le e^{-l(1-\|\cL_0(S)\|)}.
\]
In particular, the distribution of $Y_l$ is very close to uniform if say $l\ge10(1-\|\cL_0(S)\|)^{-1}\log|G|$.
More precisely, for such an $l$, we have
\[
\left|\P(Y_l=g)-\frac{1}{|G|}\right|\le\frac{1}{|G|^{10}}.
\]

There is also a combinatorial way to characterize large spectral gap.
If the spectral gap is large, then the  Cayley graph of $G$ with
respect to $S$ has a large isoperimetric constant.
If $S$ is symmetric (i.e. $s\in S$ implies $s^{-1}\in S$), then the converse is also true.
Graphs with large isoperimetric constants are called expanders.
For more details we refer to Lubotzky's survey \cite{Lub-colloquium2011}.

The problem of studying spectral gaps of random walks is interesting in its own right, and it has been
studied extensively.
Recently, spectral gap estimates were used together with sieve techniques to prove various results
in number theory and group theory.
A detailed account on these developments would go beyond the scope of our paper, so we refer
the interested reader to the recent surveys \cite{Lub-colloquium2011} and \cite{Kow-Bourbaki}.

\subsection{Statement of the result}
Let $p$ be a prime and denote by $\F_p$ the finite field of order $p$.
Our result compares the spectral gap of a random walk on $\ag$, the group of
affine transformations of $\F_p^d$ with its projection to $\SL_d(\F_p)$.

\begin{thm}\label{theorem:main}
There is a number $c$ depending only on $d$ such that the following holds.
Let $S'\subset\SL_d(\F_p)$, and let $S\subset\ag$ be such a set that for each $g\in S$
there is precisely one $\s\in S'$ such that the linear part of $g$ is $\s$.
Suppose further that $S$ is not contained in a coset of a proper subgroup of $\ag$.
Then
\[
1-\|\cL_0(\ag,S)\|\ge c\cdot\min\{1-\|\cL_0(\SL_d(\F_p),S')\|,|S|^{-1}\}.
\]
\end{thm}

Up to the constant $c$, the bound is sharp, as can be seen by the example when all but
one element of $S$ is contained in a subgroup isomorphic to $\SL_d(\F_p)$.
However, when the distribution of $S$ is better among the cosets of such subgroups
the bound can be improved.
In the next section, we will formulate a slightly more general version of this theorem
with an improved bound.

In order to apply Theorem~\ref{theorem:main}, a bound on the spectral gap
for the projection of the random walk in $\SL_d(\F_p)$ is required.
The following important result of Bourgain and Gamburd provides such a bound:

\begin{AbcTheorem}[Bourgain, Gamburd]\label{theorem:BourgainGamburdSLd}
Let $\overline S\subset\SL_d(\Z)$ be a finite symmetric set, which generates a Zariski-dense
subgroup.
Then there is a number $c$ depending on $\overline S$ (but not on $p$) such that the following
holds for all but finitely many primes $p$.
Let $S$ be the mod $p$ projection of $\overline S$.
Then
\[
1-\|\cL_0(\SL_d(\F_p),S)\|>c.
\]
\end{AbcTheorem}

The $d=2$ case of this theorem is \cite[Theorem 1]{BG-prime}, and this
has been worked out by Kowalski \cite{Kow-Sl2} with explicit
constants.
A key ingredient in the proof is Helfgott's product theorem in \cite{Hel-Sl2}.
The $d\ge 3$ case is \cite[Theorem 1.2]{BG-primepower-II} which assumes a generalization of
Helfgott's theorem as a black box.
This generalization is due to Helfgott \cite{Hel-Sl3} in the $d=3$ case, and to Breuillard, Green and Tao \cite{BGT-product-theorem} and
independently by Pyber and Szab\'o \cite{PS-product-theorem} in the general case.

A key problem in the subject highlighted in \cite{LW-expanders} is determining how the size of the spectral gap depends on the choice of generators. It seems plausible that the conclusion of Theorem~\ref{theorem:BourgainGamburdSLd} could hold
with a constant $c$ depending only on $d$ and $|S|$, and not on a previously fixed set in $\SL_d(\Z)$.
The following theorem by Breuillard and Gamburd \cite[Theorem 1.1]{BG-strong-uniform} gives some evidence in this direction:

\begin{AbcTheorem}[Breuillard, Gamburd]\label{theorem:BreuillardGamburd}
For any $\delta>0$ and positive integer $N$,
there is a constant $c>0$ depending only on $\delta$ and $N$ such that for any sufficiently large $X$, for all but $X^{\delta}$ primes $p\le X$, for any symmetric generating set $S\subset\SL_2(\F_p)$ with $|S|=N$,
\[
1-\|\cL(\SL_2(\F_p),S)\|>c.
\]
\end{AbcTheorem}

Theorem~\ref{theorem:main} above can be seen in this general context: while the spectral gap on $\ag$ is dependent on the choice of generators $S'$ for $\SL_d(\F_p)$, the estimate given by the theorem is uniform in the way $S'$ is lifted to a generating set $S$ on $\ag$.
 If one takes $S$ to be the projection mod $p$ of a fixed set
$\overline S\subset\Z^d\rtimes\SL_d(\Z^d)$ generating a (fixed) Zariski dense subgroup,
establishing a spectral gap (uniform in $p$) for the corresponding averaging operator
can be obtained by an  adaptation of the method of Bourgain and Gamburd
without introducing any substantial new ideas.
In particular, it  is a very special case of the main result of \cite[Theorem 1]{SGV-perfect}.

\subsection{Motivation}
One source of interest in the group $\ag$ stems from a continuous analogue of the problem.
In that analogue, the role of $\SL_d(\F_p)$ is played by the compact Lie group $\SO(d)$
and $\ag$ is replaced by $\R^d\rtimes\SO(d)$, the group of orientation preserving isometries of
Euclidean space.
In the paper \cite{LV-Euclidean}, we prove an analogue of Theorem \ref{theorem:main} in
that setting. This has two applications of independent interest in quite different directions:
\begin{itemize}
\item
Under the assumption that a corresponding random walk on $\SO(d)$ has spectral gap, we show that
a self-similar measure is absolutely continuous, provided the contraction coefficients of the self-similarities
are sufficiently close to 1.

\item In \cite{Var-Euclidean} a local-central limit theorem for a random walk on $\R^d$ by Euclidean isometries is proved. In \cite{LV-Euclidean} we strenghten this result when the underlying random walk on $\SO(d)$ has spectral gap to show that this local-central limit theorem holds at a scale exponentially small in the number of steps (when $d\geq3$ and the rotation part of the random walk generates a dense subgroup of $\SO(d)$, the local-central limit theorem is established in \cite{Var-Euclidean} only up to the scale $e^{-O(l^{1/3})}$ where $l$ is the number of steps).
\end{itemize}

\subsection{Ideas in the proof}

We heavily exploit the method of Bourgain and Gamburd in our proof of Theorem \ref{theorem:main}.
Note however that the product theorems of \cite{Hel-Sl2,Hel-Sl3,BGT-product-theorem,PS-product-theorem} are not used in the proof of Theorem \ref{theorem:main}, at least not directly; in their stead we use the assumed spectral gap of the averaging operator 
$\cL_0(\SL_d(\F_p),S')$ corresponding to the associated random walk on $\SL_d(\F_p)$.

We recall the essence of the method in Theorem \ref{theorem:BourgainGamburd} in Section
\ref{section:BourgainGamburd}.
To apply this theorem to the problem at hand, we need to show that the random walk does not concentrate
on cosets of subgroups isomorphic to $\SL_d(\F_p)$.
More precisely, we show that
\[
\P(Y_l\in A)\le 4p^{-d/4}
\]
if $A$ is a coset of a subgroup isomorphic to $\SL_d(\F_p)$ and $l\ge C\log p$ with a constant $C$ sufficiently
large.

This non-concentration estimate proved in Section \ref{section:decay} is the main new contribution in our paper.
It is essentially equivalent to proving a non-concentration estimate for the
random walk on $\F_p^d$ generated by $S$, which we do by showing the $L^2$-norm of the probability measure on $\F_p^d$ after $O_{d, S'}(\log p)$ many steps becomes small.

The proof of this fact rests upon an observation that if $\eta$ is an (arbitrary) probability measure
on $\F_p^d$ and the absolute value of its Fourier transform is almost constant in the appropriate sense on $\F_p^d\setminus \{0\}$, then the measure is
either spread out on $\F_p^d$ or most of the contribution to the $L^2$-norm of $\eta$ comes from a single atom;
cf.~Proposition \ref{proposition:nonconstant}.

Using this observation we argue iteratively: if the random walk on $\F_p^d$ after some steps concentrates the $L^2$-norm in a single atom, we can use that not
all elements of $S$ move this atom to the same point to show that the next step quantifiably reduces the $L^2$-norm.
If the Fourier transform does not have almost constant absolute value for all nonzero coefficients, we can use the spectral
gap for the projection to $\SL_d(\F_p^d)$, to prove that the next step quantifiably reduces the $L^4$
norm of the Fourier transform.

To carry out this argument, we have to work with both $\Ltwo$ and $\Lfour$ norms.
Therefore it will be necessary, to relate the $\Ltwo$ and $\Lfour$ spectral gaps.
This can be done using either Riesz-Thorin interpolation as we do here, or some classical results on $L^q$ spaces similarly to the paper \cite{BFGM-Lp} and the uniform convexity of $L^p$ spaces \cite[Chapter 9]{BL-book} as was done in an earlier version of this paper \cite{arxiv-version-one}.

The paper contains another new idea in Section \ref{section:growth}, where we prove a result about
growth of product-sets in the group $\ag$.
The result (Proposition \ref{proposition:growth})
itself is not new; similar statements appear in the papers \cite{SGV-perfect}
and \cite{PS-product-theorem}.
However, we give a new proof that can be adapted to work in the continuous case as done in \cite{LV-Euclidean}.

\subsection{An open problem}
An interesting analogue of the results we obtained here is provided by a random walk using a set $S$ of generators on the group $\SL_2(\F_p)\times\SL_2(\F_p)$.
Is it possible to estimate the spectral gap in terms of the spectral gaps
of the projections to the direct factors analogously to Theorem \ref{theorem:main}?

If one tries to prove such an estimate using the method of Bourgain and Gamburd then
the following problem arises.
The group $\SL_2(\F_p)\times\SL_2(\F_p)$ contains the subgroup $\{(g,g):g\in \SL_2(\F_p)\}$
and its conjugates, and it may happen that the random walk concentrates too much mass
on such a subgroup.
If this obstacle could be ruled out, then a positive solution to the above problem would follow
immediately from the method of Bourgain and Gamburd.
This difficulty is similar to the one we tackle in this paper, but its solution probably require a
different set of ideas.
We also mention that when one looks at the problem in the group $\SL_2(\F_{p_1})\times\SL_2(\F_{p_2})$
for different primes $p_1\neq p_2$, then the problem disappears, since all proper subgroups of
$\SL_2(\F_{p_1})\times\SL_2(\F_{p_2})$ projects into a proper subgroup of one of the factors.

\subsection{Organization}
In the next section, we introduce some more notation and state a more technical and somewhat stronger version
of Theorem \ref{theorem:main}.
In Section \ref{section:decay}, we prove the crucial non-concentration estimate mentioned above.
We recall the method of Bourgain and Gamburd in Section \ref{section:BourgainGamburd} and use
it to deduce our results.

\subsection*{Acknowledgment}
We are grateful to the referee and Lam Pham for carefully reading the paper and
for suggestions that significantly improved the presentation of the paper.

%%%%%%%%%%%%%%%%%%%%%%%%%%%%%%%%%%%%%%%%%%
\section{Notation}
\label{section:notation}
%%%%%%%%%%%%%%%%%%%%%%%%%%%%%%%%%%%%%%%%%%

For $(v_1,\t_1),(v_2,\t_2)\in\ag$, the product is defined by
\[
(v_1,\t_1)\cdot(v_2,\t_2)=(v_1+\t_1v_2,\t_1\cdot\t_2).
\]
If $g\in\ag$ can be written in the form $g=(v,\t)$ then we write $v(g)=v$ and $\t(g)=\t$.
In other words $v(g)$ and $\t(g)$ are the projections to the factors $\F_p^d$ and $\SL_d(\F_p)$
respectively.
We note that $v(g)$ is not intrinsically defined and we fix one choice for the entire paper.

The group $\ag$ naturally acts on $\F_p^d$ by means of the formula
$g.x=v(g)+\t(g)x$ for $g\in\ag$ and $x\in\F_p^d$.
This action is consistent with the above product law, i.e. we have the identity $(g_1\cdot g_2).x=g_1.(g_2.x)$.

It is easy to check that the inverse of an element $g\in\ag$ is given by the formula
\be\label{equation:inverse}
g^{-1}=(-\t(g)^{-1}.v(g),\t(g)^{-1}).
\ee

We identify measures on finite sets with their Radon-Nikodym derivative with respect to the counting measure.
Thus the difference between our use of the words function and measure is purely rhetoric.
With this convention we also write $f(A)=\sum_{x\in A}f(x)$, where $A$
is a finite set and $f$ is a function (or measure) defined on a finite set containing $A$.
A probability measure is a non-negative measure with total mass 1.
The Dirac delta measure concentrated at the point $x$ is a probability measure $\d_x$ such that
$\d_x(x)=1$ and $\d_x(y)=0$ if $y\neq x$.

Let $\mu$ be a measure on $\ag$.
We denote its $l$-fold convolution by
\[
\mu^{*(l)}=\underbrace{\mu*\ldots *\mu}_{l}.
\]
We denote the convolution of $\mu$ with a measure $\nu$ on $\F_p^d$ by
\[
[\mu.\nu](x)=\sum_{g\in G}\mu(g)\nu(g^{-1}.x)
\]
which is a measure on $\F_p^{d}$.

The left regular representation on $\ag$ is denoted by $\cL$ and the representation
obtained by composing the homomorphism $\t$ with the left regular representation
of  $\SL_d(\F_p)$ is denoted by $\cL^\t$.
They are defined by the formulas
\[
[\cL(g)f](h)=f(g^{-1}h)\quad{\rm and}\quad
[\cL^\t(g)f'](\s)=f'(\t(g)^{-1}\s)
\]
for $f\in\C[\ag]$, $f'\in\C[\SL_d(\F_p)]$, $g,h\in\ag$ and $\s\in\SL_d(\F_p)$.
In addition, we denote by $\cL_0$ and $\cL_0^\t$ the restrictions of $\cL$ and
$\cL^\t$ to the corresponding codimension one subspaces orthogonal to the constants.

Let $\pi$ be a representation of $\ag$ and $\mu$ a probability measure on $\ag$.
We write 
\[
\pi(\mu)=\sum_{g\in\ag}\mu(g)\pi(g)
\]
which is an operator acting on the relevant representation space.
Compare these with the definition of $\cL(G,S)$ in the previous section.

We can generalize the notion of the random walk by considering random elements
$X_l$ having an arbitrary common law $\mu$ instead of a uniform distribution on a
finite set $S$.
Note that with the above notation, the law of $Y_l$ is the probability measure
\[
\cL(\mu)^l\d_1=\mu^{*(l)}.
\]

Now we state a more general version of Theorem \ref{theorem:main}
with a slight improvement in the bound.

\begin{thm}\label{theorem:technical}
There is a constant $c$ depending only on $d$ such that the following holds.
Let $\mu$ be a probability measure on $\ag$.
Let $\a$ be the maximal probability for the event that
a random element $X\in\ag$ of law $\mu$
takes a given point $x\in\F_p^d$ to a given point $y\in\F_p^d$.
That is
\[
\a=\max_{x,y\in\F_p^d}\mu.\d_{x}(y).
\]
Then
\[
1-\|\cL_0(\mu)\|\ge c\min\{1-\|\cL_0^\t(\mu)\|,1-\a\}.
\]
\end{thm}

In the setting of Theorem \ref{theorem:main}, for every $x,y\in\F_p$ there is
at least one element $g\in S$ such that $g.x\neq y$.
Indeed, in the opposite case, $S$ would be contained in a coset of a subgroup isomorphic
to $\SL_d(\F_p)$.
Thus $\a\le1-|S|^{-1}$, and Theorem \ref{theorem:technical} indeed contains Theorem \ref{theorem:main}
as a special case.
In the rest of the paper we prove Theorem \ref{theorem:technical}.

%%%%%%%%%%%%%%%%%%%%%%%%%%%%%%%%%%%%%%%%%%
\section{Non-concentration on subgroups}
\label{section:decay}
%%%%%%%%%%%%%%%%%%%%%%%%%%%%%%%%%%%%%%%%%%

In this section, we make stronger assumptions on $\mu$ than in Theorem~\ref{theorem:technical},
but we will see in Section \ref{section:proof} that the general case can be reduced to this one.
Let $v_0\in\F_p^d$ be an arbitrary point, and
consider the sequence of probability measures $\eta_l=\mu^{*(l)}.\d_{v_0}$.
These measures can be thought of as the laws of the steps of a random walk on
$\F_p^d$ starting from the point $v_0$, and the steps being made by applying a random
element of $\ag$ with law $\mu$.

The purpose of this section is to prove the following result.
\begin{prp}\label{proposition:decay}
Suppose that $\mu$ is symmetric and
\[
\|\mu.\d_{x}\|_\Ltwo\le\frac34
\]
for all $x\in\F_p^d$.
Suppose further that
\[
\|\cL_0^\t(\mu)\|\le\frac{1}{2}.
\]
Then for $l=\lfloor 2^{15}d\log p\rfloor$, we have
\[
\|\eta_l\|_{L^\infty}\le\|\eta_l\|_\Ltwo\le4p^{-d/4}.
\]
\end{prp}

One can interpret this proposition as a non-concentration estimate.
Indeed, the set $A_{v_0,u_0}\subset\ag$ consisting of elements $g$ with the
property $g.v_0=u_0$ is a coset of a subgroup isomorphic to $\SL_d(\F_p)$.
Moreover,
\[
\P(Y_l\in A_{v_0,u_0})=\mu^{*(l)}(A_{v_0,u_0})=\eta_l(u_0),
\]
which is estimated in the proposition.

We keep all constants explicit in this section for the sake of clarity, but
we make no efforts to optimize them.

% %%%%%%%%%%%%%%%%%%%%%%%%%%%%%%%%%%%%%%%%
\subsection{Some properties of the Fourier transform}
\label{section:nonconstant}
%%%%%%%%%%%%%%%%%%%%%%%%%%%%%%%%%%%%%%%%

We introduce some notation related to the Fourier transform on $\F_p^d$
and some conventions for normalization.
We denote the vector space of complex valued functions on $\F_p^d$ by $\CF$.
Let $f\in\CF$ and define its Fourier transform by
\[
\wh f(\xi)=\sum_{x\in\F_p^d} e(\langle x,\xi\rangle)f(x),
\]
where $e(y)=e^{-2\pi iy/p}$ for $y\in \Z/p\Z=\F_p$.

We distinguish the space on which the Fourier transform is defined denoting it by $\wh\F_p^d$.
This space is of course isomorphic to $\F_p^d$, but we make this distinction in our notation because
it will be convenient for us to use different normalizations for the $L^q$ norms of functions
on these two spaces.
For functions $f\in\CF$  and $\f\in\CFh$, we define these norms by
\[
\|f\|_\Lq:=\left(\sum_{x\in\F_p^d}f(x)^q\right)^{1/q}\quad{\rm and}\quad
\|\f\|_{\Lhq}:=\left(\frac{1}{p^d}\sum_{\xi\in\wh\F_p^d}\f(\xi)^q\right)^{1/q}.
\]
With this normalization, Plancherel's formula becomes $\|f\|_\Ltwo=\|\wh f\|_\Lhtwo$.

Remove the origin $0$ from the set $\wh\F_p^d$ and denote it by $X$. Let $\iota$ be the extension map from functions on $X$ to functions on $\F_p^d$, i.e.\ $\iota (f)[x]= f (x)$ for $x \in\wh\F_p^d \setminus 0$ and $\iota (f)[0]=0$, and $\iota ^*$ the natural projection from functions on $\wh\F_p^d$ to functions on $X$(we will also  use $\iota,\iota^*$ to denote the exact same maps between $L ^ q (\F_p ^ d \setminus 0)$ and $L ^ q (\F_p ^ d)$).
The norm $\Lhq$ on $X$ is defined as $\|\f\|_{\Lhq}=\|\iota (f)\|_{\Lhq}$, i.e.\
the total mass of the measure on $X$ with respect to which the $\Lhq$-norms
are defined is $(p^d-1)/p^d$.

Let $\cA$ denote the natural actions of $\F_p^d \rtimes \SL_d(\F_p)$ on both $\Lq$ and $\Lhq$, and $\cA^\t$ the corresponding action of the linear part of $\F_p^d \rtimes \SL_d(\F_p)$ on $\Lhq$. Note that for any $f \in \Lhq$,\ $g \in \F_p^d \rtimes \SL_d(\F_p)$, and $x \in \F_p^d $
\[
\absolute{\left(\cA(g) f\right) (x)}=\left(\cA^\t(g)|f|\right)(x).
\]
It would sometimes be useful to think of $\cA^\t$ also as an action of the group $\SL_d(\F_p)$.

The purpose of this section is to prove the following estimate,
showing that if $\eta$ is a probability measure on $\F_p^d$
whose $L^2$-norm is not too concentrated at a single atom
and the nontrivial Fourier coefficients of $\eta$ are not very
small then the absolute value of the nontrivial Fourier coefficients of $\eta$ cannot be almost constant.

\begin{prp}\label{proposition:nonconstant}
Let $\eta$ be a probability measure on $\F_p^d$.
Suppose that
\[
\eta(x)\le \frac{40}{41}\| \eta \|_\Ltwo \qquad\text{for every $x \in\F_p^d$}
\]
and
\begin{equation}\label{lower bound on L4}
\norm {\wh \eta}_{\Lhfour}\ge 4p^{-d/4}.
\end{equation}
Then there is an $h \in \SL _ d (\F_p)$ so that
\[
\norm {\left(\absolute {\wh \eta}-\cA^\t(h)\absolute {\wh \eta}\right)}_{\Lhfour}\ge \frac{7}{100} \norm {\wh \eta}_{\Lhfour}.
\]
\end{prp}

A key ingredient is the use of Plancherel's formula for the measure
$\eta*\check\eta$, where $\check\eta$ denotes the probability measure on $\F_p^d$
given by the formula:
\[
\check\eta(x)=\eta(-x).
\]
Observe that the Fourier transform of $\eta*\check\eta$ is $|\wh\eta|^2$ and we need to show that it is
not too close to constant.
This is equivalent to $\eta*\check\eta$ being far from a Dirac measure supported at 0.
This latter property of $\eta*\check\eta$ is proved in the next Lemma.

\begin{lem}\label{lemma:no-mass-origin}
Let $\eta$ be a probability measure on $\F_p^d$, and
suppose that
\[
\eta(x)\le\frac{40}{41}\|\eta\|_\Ltwo\qquad\text{for every $x\in\F_p^d$.}
\]
Then
\[
\|\iota^*(\eta*\check{\eta})\|_\Ltwo^2\ge\frac{1}{42}\|\eta*\check{\eta}\|_\Ltwo^2.
\]
\end{lem}
\begin{proof}
By simple calculation:
\bea
\|\iota^*(\eta*\check{\eta})\|_\Ltwo^2&\equiv&\sum_{x\neq0}(\eta*\check{\eta})(x)^2
=\sum_{x\neq0}\left[\sum_{y,z:y-z=x}\eta(y)\eta(z)\right]^2\nonumber\\
&\ge&\sum_{x\neq0}\sum_{y,z:y-z=x}[\eta(y)\eta(z)]^2
=\sum_{y,z:y\neq z}\eta(y)^2\eta(z)^2\nonumber\\
&=&\frac{1}{2}\left[\left(\sum_{y\in\F_p^d}\eta(y)^2\right)^2-\sum_{y\in\F_p^d}\eta(y)^4\right]
\label{equation:origin1}
\eea
Using the assumption in the lemma:
\be\label{equation:origin2}
\sum_{y\in\F_p^d}\eta(y)^4
\le\max_{y\in\F_p^d}\{\eta(y)^2\}\sum_{y\in\F_p^d}\eta(y)^2
\le\frac{1600}{1681}\left[\sum_{y\in\F_p^d}\eta(y)^2\right]^2
\ee

{}From inequalities \eqref{equation:origin1} and \eqref{equation:origin2}
we get
\[
\sum_{x\neq0}(\eta*\check{\eta})(x)^2\ge\frac{1}{42}\|\eta\|_\Ltwo^4
\]
because $1/42<(1-1600/1681)/2$.
Thus
\[
\eta*\check{\eta}(0)^2=\left[\sum_{x\in\F_p^d}\eta(x)^2\right]^2=\|\eta\|_\Ltwo^4
\le 42 \sum_{x\neq0}(\eta*\check\eta)(x)^2,
\]
which implies the claim.
\end{proof}

A tool which will help us relate the $L^2$ and $L^4$ norm is the Mazur map.
We recall its definition and properties from the book of Benyamini and
Lindenstrauss \cite[Chapter 9]{BL-book}.
We denote by $S(L^q)$ the unit sphere of the space $L^q(X)$.
For $f\in S(L^4)$, the Mazur map is defined by
\[
\phi(f)=|f|^2\sign(f),
\]
where $\sign(f)=f/|f|$ for $f\neq 0$ and $0$ otherwise.

The Mazur map is a homeomorphism from  $S(L^4)$ to $S(L^2)$.
Moreover, we have the following inequalities.
\begin{AbcTheorem}[{\cite[Theorem 9.1]{BL-book}}]\label{theorem:Mazur}
For $f_1,f_2\in S(L^4)$, we have
\[
\|f_1-f_2\|_\Lfour\ge\frac{1}{2}\|\phi(f_1)-\phi(f_2)\|_\Ltwo.
\]
\end{AbcTheorem}
For proof, see \cite[Proof of Theorem 9.1]{BL-book}, applied to $p=4$ and $q=2$.

\begin{proof}[Proof of Proposition \ref{proposition:nonconstant}]
The existence of a $h \in \SL_d(\F_p)$ as in the proposition follows from the following estimate:
\begin{multline}\label{using Mazur}
\tfrac {1 }{ \#\SL_d(\F_p)} \!\sum_ {h \in\SL_d(\F_p)} \frac{\norm {\left (\absolute {\wh \eta} - \cA^\t(h)\absolute {\wh \eta} \right)} _ \Lhfour }{\norm {\wh \eta}_ \Lhfour }\\
\begin{aligned}
 & \ge
\tfrac {1 }{ 2\#\SL_d(\F_p)}\! \sum_ {h \in\SL_d(\F_p)} \frac{\norm {\left (\absolute {\wh \eta} ^2 - \cA^\t(h)\absolute {\wh \eta} ^2 \right)} _ \Lhtwo}{\norm {|\wh \eta|^2}_ \Lhtwo}  \\
&\ge \frac{\Bigl\|\absolute {\wh \eta} ^2 - \tfrac {1 }{ \#\SL_d(\F_p)}\! \sum_ {h \in\SL_d(\F_p)} \cA^\t(h)\absolute {\wh \eta} ^2 \| _ \Lhtwo}{2\norm {|\wh \eta|^2}_ \Lhtwo}\\
&\ge \frac{\norm {\iota ^{*} (\absolute {\wh \eta} ^2) - c}_{\Lhtwo(X)}}{2\norm {|\wh \eta|^2}_ \Lhtwo}
\end{aligned}
\end{multline}
for $c = (p ^ d -1) ^{-1} \sum_ {x \ne 0} \absolute {\wh \eta (x)} ^2$; in particular, $0 \le c \le 1$. Note that Theorem \ref{theorem:Mazur} was used to pass from the first to the second line in \eqref{using Mazur}.

Since $\wh \eta (0) = 1$, and using \eqref{lower bound on L4}, we have that
\begin{align*}
\norm {\iota ^{*} (\absolute {\wh \eta} ^2) - c}_{\Lhtwo(X)} ^2& = \norm {\absolute {\wh \eta} ^2 - c}_{\Lhtwo} ^2 - p ^ {- d} (\absolute {\wh \eta(0)} ^2 - c) ^2 \\
& \ge \norm {\absolute {\wh \eta} ^2 - c}_{\Lhtwo} ^2 - p ^ {- d} \\
& = \norm {\eta*\check{\eta}-c \delta _ 0} _ {L ^2} ^2 - p ^ {- d} \\
& \ge  \norm {\iota ^{*} (\eta*\check{\eta})} _ {L ^2} ^2 - p ^ {- d} \\
& \ge \left (\frac {1 }{ 42} - \frac {1 }{ 256} \right) \norm {\eta*\check{\eta}} _ {L ^2} ^2
\end{align*}
since $\norm {\eta*\check{\eta}} _ {L ^2} = \norm {\absolute {\wh \eta} ^2}_{\Lhtwo} = \norm {\wh \eta}_ \Lhfour^2 \ge 16 p^{-d/2}$.
In the last line, we used Lemma \ref{lemma:no-mass-origin}.
Using equation \eqref{using Mazur} we can now conclude that there is some $h$ so that
\begin{equation*}
\frac{\norm {\left (\absolute {\wh \eta} - \cA^\t(h)\absolute {\wh \eta} \right)} _ \Lhfour }{\norm {\wh \eta}_ \Lhfour } \ge \frac {\sqrt {1/42-1/256}}{2} \ge \frac {7 }{ 100}
\end{equation*}
establishing the proposition.\end{proof}

%%%%%%%%%%%%%%%%%%%%%%%%%%%%%%%%%%%%
\subsection{A consequence of the Riesz-Thorin interpolation theorem}
We will now use the Riesz Thorin interpolation theorem to study how~${\cA^\t}(\mu)$ acts on the Fourier transform of measures on $\F_p ^ d$ with respect to the ${\Lhfour}$-norm.

\begin{prp}\label{prp:L4 decay} Let $\mu$ be a probability measure on $\SL _ d (\F_p)$ satisfying the conditions of Proposition~\ref{proposition:decay} and $\eta$ a probability measure on $\F_p ^ d$ satisfying the conditions in Proposition~\ref{proposition:nonconstant}. Then
\begin{equation*}
\norm {{\cA^\t} (\mu^{*(5)}) \absolute {\wh\eta}} _ {{\Lhfour}} \le e^{-5\cdot 2^{-14}}\norm{\wh\eta}_{\Lhfour}
.\end{equation*}
\end{prp}

\begin{lem}\label{lemma: uniform convexity} Let $f, g$ be nonnegative functions on a $\sigma$-finite measure space. Then
\begin{equation*}
{\tfrac{1}{2}} (\norm f ^ 4 _ {L ^ 4} + \norm g ^ 4 _ {L ^ 4} )\ge \norm {\frac {f + g }{ 2}} ^ 4 _ {L ^ 4} +7 \norm {\frac {f - g }{ 2}} ^ 4 _ {L ^ 4}
.\end{equation*}
\end{lem}
\begin{proof}
This follows easily from the inequality
\begin{equation*}
\frac {1+x^4 }{ 2} \ge \left (\frac {1+x }{ 2} \right)^4+7\left (\frac {1-x }{ 2} \right)^4
\end{equation*}
which is valid for $x \ge 0$.
\end{proof}

\begin{proof}[Proof of Proposition~\ref{prp:L4 decay}]
Consider for $h \in \SL _ d (\F_p)$ the operator \[({\cA^\t} (h) -1) {\cA^\t} (\mu ^ {*(10)}),\]
where $1$ denotes the identity operator.
The assumption $\norm {{\cL_0^\t} (\mu)}_{\Ltwo} \le {\tfrac{1}{2}}$ implies
$\norm {{\cA_0^\t} (\mu)}_{\Ltwo} \le {\tfrac{1}{2}}$.
Since $\cA^\t(h)-1$ annihilate the constants and it has $L^2$ and $L^\infty$ norm at most $2$,
we have that
\begin{align*}
\norm {({\cA^\t} (h) -1) {\cA^\t} (\mu ^ {*(10)})}_{\Lhtwo} &\le 2^{-9}\\
\norm {({\cA^\t} (h) -1) {\cA^\t} (\mu ^ {*(10)})}_{\wh L^\infty} &\le 2.
\end{align*}
Hence by interpolation
\begin{equation*}
\norm {({\cA^\t} (h) -1) {\cA^\t} (\mu ^ {*(10)})}_{\Lhfour} \le 2^{-4}.
\end{equation*}
For any $h \in \SL _ d (\F_p)$,
\begin{align*}
\norm {{\cA^\t} (h) \absolute {\wh \eta} -\absolute {\wh \eta}}_{\Lhfour}&
\le
\Bigl\|{\cA^\t} (h)  {\cA^\t} (\mu ^ {*(10)})\absolute {\wh \eta}
-{\cA^\t} (\mu ^ {*(10)})\absolute {\wh \eta}\Bigr\|_{\Lhfour} +\\
&\qquad+ 2 \norm {{\cA^\t} (\mu ^ {*(10)}) \absolute {\wh \eta} - \absolute {\wh \eta}} _ {\Lhfour}\\
&\le
2^{-4} \norm {\wh\eta}_{\Lhfour}
+ 2 \norm {{\cA^\t} (\mu ^ {*(10)}) \absolute {\wh \eta} - \absolute {\wh \eta}} _ {\Lhfour}
\end{align*}
hence using Proposition~\ref{proposition:nonconstant}
there is a $h \in \SL _ d (\F_p)$ so that
\begin{align}\label{the 43 equation}
\norm {{\cA^\t} (\mu ^ {*(10)}) \absolute {\wh \eta} - \absolute {\wh \eta}} _ {\Lhfour} &\ge
\frac 12
\Bigl (\frac{7}{100}-\frac{1}{16}\Bigr) \norm {\wh \eta} _ {\Lhfour} = \frac{3}{800} \norm {\wh \eta} _ {\Lhfour}
.\end{align}
As $\mu$ is symmetric,
\begin{multline}\label{using symmetry}
\norm {{\cA^\t} (\mu ^ {*(10)}) \absolute {\wh \eta} - \absolute {\wh \eta}} _ {\Lhfour}\\  \le\sum_ {g, g '} \mu ^ {*(5)} (g) \mu ^ {*(5)} (g ') \norm {({\cA^\t} (g) - {\cA^\t} (g ')) \absolute {\wh \eta}} _ {\Lhfour}
.\end{multline}
But by Lemma~\ref{lemma: uniform convexity}, for any nonnegative $f \in {\Lhfour}$
\begin{equation*}
\norm {\frac {({\cA^\t} (g) + {\cA^\t} (g ')) f }{ 2}} _ {\Lhfour} ^ 4 \le \norm {f} _ {\Lhfour} ^ {4} - \frac {7 }{ 16} \norm {({\cA^\t} (g)  - {\cA^\t} (g '))f} _ {\Lhfour}^4
,\end{equation*}
hence as $(1-x)^{1/4} \le 1 - x/4$ for $0 \le x \le 1$
\begin{equation}\label{the 7/256 equation}
\norm {\frac {({\cA^\t} (g) + {\cA^\t} (g ')) f }{ 2}} _ {\Lhfour} \le \norm {f} _ {\Lhfour}  - \frac {7 }{ 64} \norm {({\cA^\t} (g)  - {\cA^\t} (g '))f} _ {\Lhfour}
.\end{equation}
Applying \eqref{the 43 equation}, \eqref{using symmetry} and \eqref{the 7/256 equation} it follows that
\begin{multline*}
\norm{{\cA^\t}(\mu ^ {*(5)})\absolute {\wh \eta}}_{\Lhfour} \le \sum_ {g, g '}  \mu ^ {*(5)} (g) \mu ^ {*(5)} (g ') \norm {\left(\frac{{\cA^\t} (g) + {\cA^\t} (g ')}2\right) \absolute {\wh \eta}} _ {\Lhfour}\\
\begin{split}
& \le
\norm {\wh \eta} _ {\Lhfour} - \frac {7}{64} \sum_ {g, g '}  \mu ^ {*(5)} (g) \mu ^ {*(5)} (g ') \norm {({\cA^\t} (g) - {\cA^\t} (g ')) \absolute {\wh \eta}} _ {\Lhfour} \\
& \le \left (1 - \frac {21}{ 64\cdot 800} \right) \norm{\wh \eta} _ {\Lhfour} \\
& \le e^{- 5\cdot 2^{-14}} \norm{\wh \eta} _ {\Lhfour}
.\end{split}
\end{multline*}
\end{proof}

\subsection{Finishing the proof of Proposition~\ref{proposition:decay}}

We consider two cases.
First we suppose that $\eta_l$ does not concentrate too big mass
on a single atom, that is:
\be\label{equation:case1}
\eta_l(x)\le\frac{40}{41}\|\eta_l\|_\Ltwo
\ee
for all $x\in\F_p^{d}$.
If this is the case,  we can apply Proposition \ref{prp:L4 decay}
and conclude either
\begin{equation*}
\norm{\wh\eta_{l+5}}_{\Lhfour}\le \norm {{\cA^\t} (\mu^{*(5)}) \absolute {\wh\eta}} _ {{\Lhfour}} \le e^{-5\cdot 2^{-14}}\norm{\wh\eta}_{\Lhfour}
.\end{equation*}
or $\|\wh\eta_l\|_\Lhfour\le4p^{-d/4}$.

On the other hand, we always have the trivial inequality $\|\wh\eta_{l+1}\|_\Lhfour\le\|\wh\eta_l\|_\Lhfour$.
Thus if we have $\|\wh\eta_k\|_\Lhfour\ge 4p^{-d/4}$ for some integer $k>0$, then
there are at most
\[
2^{14}d\log p
\]
many nonnegative integers $l<k$ such that \eqref{equation:case1} holds.

Now we turn to the second case, when \eqref{equation:case1} does not hold, that
is, there is a point $x_0\in\F_p^d$ such that
\[
\eta_{l}(x_0)\ge\frac{40}{41}\|\eta_l\|_{\Ltwo}.
\]
In this case $\eta_l$ is very close to a constant multiple of $\d_{x_0}$ in the $\Ltwo$ norm
so we can estimate $\|\eta_{l+1}\|_{\Ltwo}$ using the assumption $\|\mu.\d_{x_0}\|_\Ltwo\le3/4$.

More precisely, we can write
\bean
\|\eta_{l+1}\|_{\Ltwo}
&=&\|\mu.\eta_l\|_{\Ltwo}\\
&\le&\frac34\eta_l(x_0)+\sqrt{\|\eta_l\|^2_\Ltwo-\eta_l^2(x_0)}\\
&\le&\left(\frac34+\frac{9}{41}\right)\|\eta_l\|_\Ltwo
<e^{-2^{-6}}\|\eta_l\|_\Ltwo.
\eean

Since $\eta_k$ is a probability measure for all integers $k\ge0$, we have $\|\eta_k\|_\Ltwo\ge p^{-d/2}$.
Therefore it follows that the number of
nonnegative integers $l<k$
such that \eqref{equation:case1}
fails is at most
\[
2^5d\log p.
\]

If we combine this with the estimate for the number of steps when \eqref{equation:case1}
holds, we can conclude that for
\[
k=\lfloor 2^{15}d\log p\rfloor
\]
we have $\|\wh\eta_k\|_\Lhfour\le 4p^{-d/4}$, hence $\|\eta_k\|_{L^2}=\|\wh\eta_k\|_\Lhtwo\le 4p^{-d/4}$.

%%%%%%%%%%%%%%%%%%%%%%%%%%%%%%%%%%%%%%%%%%%%%%%%%%%%
\section{The Bourgain-Gamburd method}
\label{section:BourgainGamburd}
%%%%%%%%%%%%%%%%%%%%%%%%%%%%%%%%%%%%%%%%%%%%%%%%%%%

We use the method of
Bourgain and Gamburd to prove Theorem \ref{theorem:technical}.
This material is fairly standard now, and most ideas have already appeared in
earlier works.
This method proves that a random walk has spectral gap if two
conditions hold.
First, the group $G$ should have no low dimensional representations.
Second, if there is a subset $A$ of $G$ of size approximately $|G|^\b$
that does not grow under multiplication, then the probability that the random walk
hits this set after approximately $l=\log|G|$ steps should be very small, e.g.
$|G|^{\e}$.

The method uses the notion of product sets, which we define now.
Let $A\subset G$ be a set; its $l$-fold product set is the set
\[
\Pi_lA=\{a_1\cdots a_l:a_1,\ldots a_l\in A\}.
\]

The method can be summarized in the following theorem.

\begin{AbcTheorem}[Bourgain, Gamburd]\label{theorem:BourgainGamburd}
There is an absolute constant $C$, and for any $\e>0$ there is a $\d>0$
such that the following holds.
Let $G$ be a finite group and $\pi$ an irreducible unitary representation of it.
Let $\mu$ be a symmetric probability measure on $G$.
Let $l_1>0$ be an integer and 
suppose that for any symmetric set $A\subset G$ that satisfies
\be\label{equation:largeA}
\mu^{*(l)}(A)\ge |G|^{-\e}
\ee
for some integer $l\ge l_1$, we either have
\be\label{equation:growth}
|\Pi_3A|\ge  |G|^\e\cdot |A| \quad{\rm or}\quad |A|\ge(\dim\pi)^{-1/3}|G|.
\ee
Then
\[
\|\pi(\mu)\|<(C\dim \pi)^{-\d/l_1}.
\]
\end{AbcTheorem}

Note that if $\e$ is too large or $\dim\pi$ too small there may be no probability measure
$\mu$ satisfying the conditions of the theorem.
Indeed, if $\lfloor (\dim\pi)^{-1/3}|G|\rfloor > |G|^{1-\e}$ then for any probability
measure $\mu$ we can find a set $A$ with
\[
(\dim\pi)^{-1/3}|G| > |A| > |G|^{1-\e}
\]
so that $\mu(A) > |G|^{-\e}$, violating the conditions of the theorem since clearly $|\Pi_3A|\le  |G|$.

This theorem is implicitly contained in the paper \cite{BG-prime}, and
variants of its proof appeared in many papers.
In particular, \cite[Corollary 4.4]{Kow-Sl2} contains a version with explicit constants,
but unfortunately, as it is stated that version only applies to groups without large normal subgroups.
For completeness, we include the proof of Theorem \ref{theorem:BourgainGamburd}
in Section \ref{section:BGproof}.

We comment on the role of the triple product set $\Pi_3A$ in the theorem.
The following Lemma shows that if the set $\Pi_k A$ is much larger than $A$, then so is $\Pi_3A$.
Hence an equivalent theorem could be stated with $\Pi_kA$ instead of $\Pi_3A$ for any integer $k\ge3$.
This observation will be important for us, since in our proof of \eqref{equation:growth},
we will estimate the size of $\Pi_{29}A$.
Note, however, that for nonabelain groups it may well happen that $\Pi_2A$ is of comparable size to $A$ but $\Pi_3A$ is much bigger.
The lemma below (in a less explicit form) is due to Tao \cite[Lemma 3.4]{Tao-noncommutative}; in this form, which is a simple corollary of the Ruzsa Triangle Inequality (cf.\ \cite[Lemma 3.2]{Tao-noncommutative} and the references given there), it can be found e.g.\ in \cite{LV-book}.

\begin{AbcLemma}\label{lemma:tripleproduct}
Let $A\subset G$ be a symmetric subset of a group.
Then for any integer $k\ge 3$, we have
\[
\frac{|\Pi_k A|}{|A|}\le\left(\frac{|\Pi_3 A|}{|A|}\right)^{k-2}.
\]
\end{AbcLemma}

The following lemma gives a lower bound on the dimension of non-trivial representations
of $\ag$.

\begin{AbcLemma}[Landazuri, Seitz]\label{lemma:dimension}
If $\pi$ is a nontrivial representation of $\ag$, then
\[
\dim\pi\ge 
\begin{cases}
\frac12(p-1)& \text{if $d=2$}\\
p^{d-1}-1&\text{otherwise.}
\end{cases}
\]
\end{AbcLemma}
\begin{proof}
Since $\ag$ is generated by subgroups isomorphic to $\SL_d(\F_p)$, the restriction of $\pi$
to one of these must be non-trivial.
Then the bound claimed in the lemma is in  \cite[p. 419]{LS-minimal-degree}.
\end{proof}

In Section \ref{section:growth}, we show that if $\mu$ is a measure that satisfy
the conditions in Proposition \ref{proposition:decay}, then any set $A$ that satisfies
\eqref{equation:largeA}, also satisfies the growth condition \eqref{equation:growth}.

In Section \ref{section:proof}, we construct a measure $\mu_0$ which satisfies the conditions in Proposition \ref{proposition:decay}
using the measure $\mu$ from Theorem \ref{theorem:technical}.
We will relate the random walks generated by the measures $\mu$ and $\mu_0$ and conclude the
proof of Theorem \ref{theorem:technical}.

%%%%%%%%%%%%%%%%%%%%%%%%%%%%%%%%%%%%%%%%%%%%%%%%%%
\subsection{Growth of product sets}
\label{section:growth}
%%%%%%%%%%%%%%%%%%%%%%%%%%%%%%%%%%%%%%%%%%%%%%%%%%

The following result is not new, a version with different constants could be deduced from the more general
results  \cite[Theorem 7]{PS-product-theorem} or \cite[Proposition 27]{SGV-perfect}.
Since this special case is much simpler, we provide a quick proof for completeness.

\begin{prp}\label{proposition:growth}
Let $\mu$ be a symmetric probability measure on $G$ that satisfies
the conditions required in Proposition \ref{proposition:decay}.
Then there is an $\e>0$ depending only on $d$, such that the following holds.
Let  $A\subset \ag$ be a symmetric set
that satisfies
\be\label{equation:largeA2}
\mu^{*(l)}(A)\ge |\ag|^{-\e}
\ee
for some integer $l\ge l_1=\lfloor 2^{15}d^2\log p\rfloor$.
Then
\[
\Pi_{29}A=\ag.
\]
\end{prp}

In what follows, we assume that $\mu$ satisfies the conditions of Proposition \ref{proposition:decay}
and $A\subset\ag$ is a set that satisfies the conditions of Proposition \ref{proposition:growth}.

We first show that $\Pi_3A$ projects onto $\SL_d(\F_p)$.
To this end, we exploit the assumption of the spectral gap in the quotient.
Then we show that there is a pure translation in $\Pi_7A$.
We conjugate this with elements of $\Pi_3A$, to get all pure translations in $\Pi_{26}A$.
Finally we multiply this with $\Pi_3A$ to recover the whole group.

The same strategy was employed in \cite{SGV-perfect}, but our proof differs in
the way we produce the first pure translation (proof of Lemma \ref{lemma:translation} below).
In \cite{SGV-perfect} the inequality $|\Pi_4A|>|\Pi_3A|$ was exploited 
(this inequality holds if one knows as is the case in \cite{SGV-perfect}
that $A$ is generating, unless of course if $\Pi_3A$ is already everything),
which implies that $\Pi_4A$ must contain two elements with the same linear part.
In the present paper, we give a different proof based on an averaging argument, which works
well in the continuous setting of \cite{LV-Euclidean} as well.

\begin{lem}\label{lemma:linearpart}
We have $\t(\Pi_3A)=\SL_d(\F_p)$.
\end{lem}
\begin{proof}
We will show that
\[
|\t(A)|\ge\frac{|\SL_d(\F_p)|}{D^{1/3}},
\]
where $D$ is the minimal dimension of a non-trivial representation of $\SL_d(\F_p)$.
Then the claim $\Pi_3\t(A)=\SL_d(\F_p)$ follows from a theorem of  Nikolov and Pyber
\cite[Corollary 1]{NP-product} (based on a paper of 
Gowers \cite{Gow-quasirandom}).

We begin by noting the identity
\[
\t(\mu^{*(l)})=\cL^\t(\mu)^l\d_{1}.
\]
By the assumption in Proposition \ref{proposition:decay},
we have $\|\cL_0^\t(\mu)\|\le1/2$.
We can write $\d_{1}=\f_1+\f_2$, such that
$\f_1\equiv 1/|\SL_d(\F_p)|$, and $\f_2$ is orthogonal to the constant.
Then
\[
\|\t(\mu^{*(l)})\|_2\le|\SL_d(\F_p)|^{-1/2}+\frac{1}{2^l}\le 2|\SL_d(\F_p)|^{-1/2},
\]
since $l\ge l_1\ge d^2\log p/\log 2$.

By the assumption in Proposition \ref{proposition:growth}, we have
\[
\sum_{g\in A}\mu^{*(l)}(g)\ge|\ag|^{-\e}.
\]
By the Cauchy-Schwartz inequality,
\bean
|\ag|^{-\e}&\le&\sum_{\s\in \t(A)}\t(\mu^{*(l)})(\s)\\
&\le&|\t(A)|^{1/2}\|\t(\mu^{*(l)})\|_2.
\eean
Combining with the inequality in the previous paragraph,
this implies
\[
|\t(A)|\ge \frac{|\SL_d(\F_p)|}{4|\ag|^{2\e}}.
\]

To finish, we note that any non-trivial representation of $\SL_d(\F_p)$
is of dimension $\ge (p^{d-1}-1)/2$  (see Lemma \ref{lemma:dimension}),
and $|\ag|\le p^{d^2+d}$.
Now the lemma follows from the remarks at the beginning of the proof,
if $\e$ is sufficiently small depending on $d$.
If $p$ is sufficiently large, any $\e\le1/(6d+6)$ works.
\end{proof}

\begin{lem}\label{lemma:translation}
There is a non-zero pure translation in $\Pi_7A$, that is, there is an element
$g_0\in \Pi_7A$, such that $\t(g_0)=1$ and $v(g_0)\neq 0$.
\end{lem}
\begin{proof}
By Lemma \ref{lemma:linearpart}, there is a map $F:\SL_d(\F_p)\to\Pi_3 A$
such that $\s=\t(F(\s))$ for all $\s\in\SL_d(\F_p)$.
We define
\[
v_0=\sum_{\s\in\SL_d(\F_p)} v(F(\s)).
\]

We show that $v_0$ is not a fixed point for all elements of $A$
under the natural action.
To this end, we write
\[
\eta_l=\mu^{*(l)}.\d_{v_0}.
\]
Since $l\ge 2^{15}d\log p$, we can apply
Proposition \ref{proposition:decay}, and we have
\[
\|\eta_l\|_{L^\infty}\le\|\eta_l\|_\Ltwo\le 4p^{-d/4}.
\]
Denoting by $G_{v_0}\subset\ag$  the stabilizer of the point $v_0\in\F_p^d$,
we have
\[
\mu^{*(l)}(G_{v_0})=\eta_l(v_0)\le4p^{-d/4}.
\]
If $\e$ is small enough, the assumption in Proposition \ref{proposition:growth}
implies that $A\not\subset G_{v_0}$.
That is, there is $g_1\in A$ such that $g_1.v_0\neq v_0$ as claimed.

We look at elements of the following form:
\[
F_2(\s)=F(\t(g_1)\s)^{-1}g_1F(\s)\in\Pi_7 A.
\]
By the definition of $F$, we have
\[
\t(F_2(\s))=(\t(g_1)\s)^{-1}\t(g_1)\s=1
\]
for all $\s\in\SL_d(\F_p)$.

On the other hand
\[
v(F_2(\s))=F_2(\s).0=-\s^{-1}\t(g_1)^{-1}.v(F(\t(g_1)\s))+\s^{-1}\t(g_1)^{-1}g_1.v(F(\s)).
\]
(To see this, recall formula \eqref{equation:inverse} for the inverse of an element of $\ag$.)
Then
\[
\t(g_1)\s.v(F_2(\s))=-v(F(\t(g_1)\s))+g_1.v(F(\s)).
\]

Since left multiplication by $\t(g_1)$ is a permutation on $\SL_d(\F_p)$, we get
\begin{align}\label{equation:average1}
\sum_{\s\in \SL_d(\F_p)}&\t(g_1)\s.v(F_2(\s))\\
&=\sum_{\s\in \SL_d(\F_p)}[-v(F(\t(g_1)\s))+g_1.v(F(\s))]\nonumber\\
&=-v_0+g_1.v_0.\label{equation:average2}
\end{align}
If $v(F_2(\s))$ were 0 for all $\s$, then \eqref{equation:average1}
would be 0.
On the other hand  \eqref{equation:average2} is clearly non-zero,
by the choice of $g_1$.
This proves that for some choice of $\s\in\SL_d(\F_p)$, $g_0=F_2(\s)$ satisfies the claims
of the lemma.
\end{proof}

\begin{proof}[Proof of Proposition \ref{proposition:growth}]
We consider the element $g_0\in\Pi_7 A$ found in Lemma \ref{lemma:translation} and
all elements of the form $gg_0g^{-1}\in\Pi_{13} A$ for $g\in\Pi_3 A$.
Since $\t(gg_0g^{-1})=\t(g)\t(g)^{-1}=1$, all of these are pure translations.
On the other hand, $v(gg_0g^{-1})=\t(g).v(g_0)$, and $\t(\Pi_3A)=\SL_d(\F_p)$,
hence $\Pi_{13}A$ contains all non-zero pure translations.
Then it follows that $\Pi_{26}A$ contains all pure translations.

Therefore, for a fixed $g\in\ag$, the set $\Pi_{26}A\cdot g$ contains all elements of
$\ag$ whose linear part is $\t(g)$.
Since $\t(\Pi_3A)=\SL_d(\F_p)$, $\Pi_{29}A=\ag$, as claimed.
\end{proof}

%%%%%%%%%%%%%%%%%%%%%%%%%%%%%%%%%%%%%%%%%%%%%%%%%%
\subsection{Proof of Theorem \ref{theorem:technical}}
\label{section:proof}
%%%%%%%%%%%%%%%%%%%%%%%%%%%%%%%%%%%%%%%%%%%%%%%%%%

We fix an integer
\[
l_0\ge\max\left\{\frac{3}{1-\a},\frac{\log 2}{2-2\|\cL_0^\t(\mu)\|}\right\}
\]
and set
\[
\mu_0=(\check\mu*\mu)^{*(l_0)},
\]
where $\check\mu$ is the measure on $\ag$ defined by
\[
\check\mu(g)=\mu(g^{-1}).
\]

The next lemma shows that the conditions of Propositions \ref{proposition:decay}
and \ref{proposition:growth} hold for $\mu_0$.
\begin{lem}\label{lemma:mu0}
With the notations above, the following holds:
\bean
\|\cL_0^\t(\mu_0)\|&\le&\frac{1}{2},\quad{\rm and}\\
\|\mu_0.\d_x\|_\Ltwo&\le&\frac34\quad{\rm for\: all\:} x\in\F_p^d.
\eean
\end{lem}
\begin{proof}
The first claim follows from
\[
\|\cL_0^\t(\mu_0)\|=\|\cL_0^\t(\check\mu*\mu)\|^{l_0}=\|\cL_0^\t(\mu)\|^{2l_0}
\le e^{(\|\cL_0^\t(\mu)\|-1)2l_0}
\]
and the assumption $l_0\ge\log2/(2-2\|\cL_0^\t(\mu)\|)$.

We turn to the proof of the second claim.
For $x,y\in\F_p^d$ and a positive integer $l$, we write
\[
\a_l(x,y)=(\check\mu*\mu)^{*(l)}.\d_{y}(x).
\]
This is the probability that the random walk on $\F_p^d$ generated by $\check\mu*\mu$ started from $y$
is at the point $x$ after $l$ steps.
It is easy to verify the identity
\[
\a_{l+1}(x,y)=\sum_{z\in\F_p^d}\a_1(x,z)\a_l(z,y).
\]

Write
\[
\a_l=\max_{x,y\in\F_p^d}\a_{l}(x,y),
\]
and observe that $\a_1\le\a$.
We claim that $\a_{l_0}\le9/16$.

To show this, write
\bean
\a_{l+1}(x,y)&=&\sum_{z\in\F_p^d}\a_1(x,z)\a_l(z,y)\\
&\le&\max_z\a_1(x,z)\cdot\max_z\a_l(z,y)\\
&&+(1-\max_z\a_1(x,z))\cdot(1-\max_z\a_l(z,y)).
\eean

In the domain $1/2\le s\le 1$, $1/2\le t \le 1$, the function $st-(1-s)(1-t)$ is monotone increasing
in both variables.
Thus
\be\label{equation:alplus1}
\a_{l+1}(x,y)\le\a_1\a_l+(1-\a_1)(1-\a_l)
\ee
provided
\be\label{equation:ge12}
\max_z\a_1(x,z)\ge\frac{1}{2}\quad{\rm and}\quad
\max_z\a_l(z,y)\ge\frac{1}{2}.
\ee

If \eqref{equation:ge12} fails, then
\[
\a_{l+1}(x,y)\le\min\{\max_z\a_1(x,z),\max_z\a_l(z,y)\}\le\frac12,
\]
so in either case we get
\[
\a_{l+1}\le\max\{\a_1\a_l+(1-\a_1)(1-\a_l),1/2\}.
\]

If $\a_{l}\le 1/2$ for some $l\le l_0$, then there is nothing to prove, so we assume this is not the case.
We can then write
\[
\Big(\a_{l+1}-\frac{1}{2}\Big)\le\a_1\Big(\a_l-\frac12\Big) -\frac12(1-\a_1)+(1-\a_1)(1-\a_l)
\le \a_1\Big(\a_l-\frac12\Big)
\]
for any $l< l_0$.
By iteration, and using $\a_1\leq \a$, we get
\[
\a_{l_0}\le \frac 12+ e^{(\a-1)l_0}.
\]
Since we took $l_0\ge 3/(1-\a)$, this implies $\a_{l_0}\le9/16$, as claimed.

To finish the proof of the second claim of the lemma,
we observe that
\[
\mu_0.\d_x(y)=\a_{l_0}(y,x)\le\a_{l_0}\le\frac{9}{16}.
\]
This implies
\[
\|\mu_0.\d_x\|_\Ltwo^2\le\|\mu_0.\d_x\|_{L^{\infty}}\cdot\|\mu_0.\d_x\|_{L^{1}}\le\frac{9}{16},
\]
which was to be proved.
\end{proof}

We are now in a position to finish the proof of Theorem \ref{theorem:technical}.
By Lemma \ref{lemma:mu0}, the conditions of Propositions \ref{proposition:decay}
and \ref{proposition:growth} are satisfied for $\mu_0$.
By these propositions and Lemma \ref{lemma:tripleproduct},
we can apply Theorem \ref{theorem:BourgainGamburd}
with $l_1=\lfloor 2^{15}d^2\log p\rfloor$ and some $\e>0$ small enough depending on $d$.
Thus we can conclude for all irreducible representations of $\ag$ that
\[
\|\pi(\mu_0)\|<(C\dim\pi)^{-\d/l_1}
\]
($\delta$ depending on $\e$, hence on $d$).
Then by Lemma \ref{lemma:dimension} we  have 
\[
\|\pi(\mu_0)\|<C_1^{1/\log p}e^{-2^{-15}\d/d}
\]
if $\pi$ is non-trivial
with $C_1$ depending only on $d$.
If $p$ is sufficiently large depending on the constants in the above inequality (hence only on $d$),
then we can write
\[
\|\pi(\mu_0)\|\le e^{-c_d},
\]
for some number $c_d>0$ depending only on $d$.
Since there are only finitely many not large enough primes, and the set of probability measures
satisfying the conclusions of Lemma \ref{lemma:mu0} is compact, the above inequality
holds for all $p$ for some number $c_d$.

Note that $\|\cL_0(\mu_0)\|$ is the maximum of $\|\pi(\mu_0)\|$ for $\pi$ running through
the non-trivial irreducible representations.
Thus
\[
\|\cL_0(\mu)\|=\|\cL_0(\mu_0)\|^{\frac{1}{2l_0}}\le e^{-\frac{c_d}{2l_0}},
\]
which was to be proved.

\subsection{Proof of Theorem \ref{theorem:BourgainGamburd}}
\label{section:BGproof}

We suppose that the assumptions of the theorem hold for some $G,\pi,\mu,\e,l_1$
and prove the conclusion for some $C,\d$.

The proof due to Bourgain and Gamburd consists of two parts.
First, we consider the $L^2$-norms $\|\mu^{*(l)}\|_2$ for $l\ge l_1$ and give improved bounds as $l$
increases.
Second, we exploit the fact that the eigenvalues of convolution operators on $L^2(G)$
have high multiplicities and hence we can get an estimate on them when $\|\mu^{*(l)}\|_2$
is close to the optimal bound, that is $|G|^{-1/2}$.
This second idea goes back to Sarnak and Xue \cite{SX-multiplicities}.

We recall the ``$L^2$-flattening Lemma'' of Bourgain and Gamburd.
This appeared implicitly in \cite{BG-prime}, and it is an application of the Balog-Szemer\'edi-Gowers
theorem combined with some results of Tao \cite{Tao-noncommutative}.
We use the version in \cite[Lemma 15]{Var-squarefree}.

\begin{AbcLemma}[Bourgain, Gamburd]\label{lemma:flattening}
Let $\nu_1$ and $\nu_2$ be two probability measures on a finite group $G$ and let $K>2$ be a number.
If
\[
\|\nu_1*\nu_2\|_2\ge\frac{\|\nu_1\|_2^{1/2}\|\nu_2\|_2^{1/2}}{K}
\]
then there is a symmetric set $S\subset G$ with
\begin{gather}
\frac{1}{CK^C\|\nu_1\|_2^2}\le  |S|\le \frac{CK^C}{\|\nu_1\|_2^2},\label{eq:Sbounds}\\
|\Pi_3 S|  \le CK^C|S|, \qquad\min_{g\in S}(\check\nu_1*\nu_1)(g) \ge\frac{1}{CK^C|S|},\nonumber
\end{gather}
where $C$ is an absolute constant.
\end{AbcLemma}

We now prove Theorem \ref{theorem:BourgainGamburd}.
Fix a number $K$ in such a way that $CK^C\le|G|^\e$, where $C$ is from Lemma
\ref{lemma:flattening} and $\e$ is from Theorem \ref{theorem:BourgainGamburd}.
(We may assume that $|G|$ is larger than any absolute constant,
since the theorem is vacuous when $|G|$ is small, if we set
the constant $C$ large enough in the theorem.)

By Lemma \ref{lemma:flattening}, for all $l\ge l_1$, we have either $\|\mu^{*(2l)}\|_2\le\|\mu^{*(l)}\|_2/K$, or
there is a symmetric set $S\subset G$ such that $|\Pi_3 S|\le {|G|}^{\e}|S|$ and
\[
\mu^{*(2l)}(S)=\check\mu^{*(l)}*\mu^{*(l)}(S)\ge\frac{|S|}{CK^C|S|}\ge|G|^{-\e}.
\]
In the latter case, $S$ satisfies condition \eqref{equation:largeA} in the theorem
and fails the first alternative of \eqref{equation:growth}.
Then we must have $|S|\ge(\dim\pi)^{-1/3}|G|$
in this case and hence by \equ{eq:Sbounds}
\[\|\mu^{*(l)}\|_2^2\le |G|^{\e}(\dim\pi)^{1/3}|G|^{-1}.
\]
We have already noted that the assumption of Theorem \ref{theorem:BourgainGamburd}
may hold only if $|G|^\e\le2(\dim\pi)^{1/3}$.
Hence in this second case we must have that $\|\mu^{*(l)}\|_2^2\le 2(\dim\pi)^{2/3}|G|^{-1}$.

We consider the sequence $a_k:=\|\mu^{*(2^kl_1)}\|_2^2$ for $k=0,1,\ldots$.
The argument of the previous paragraph shows that either $a_{k+1}\le a_k/K^2$ or
$a_k\le  2(\dim\pi)^{2/3}|G|^{-1}$.
There is an integer $L$ depending only on $\e$ such that $K^{2L}>|G|$.
Then $\|\mu^{*(2^Ll_1)}\|_2^2=a_{L}\le 2(\dim\pi)^{2/3}|G|^{-1}$.

Set $\mu_1=\mu^{*(2^Ll_1)}$, and consider the operator $T:f\mapsto\mu_1^{*(2)}*f$
acting on $L^2(G)$.
We compute the trace of $T$.
Recall that $\d_g$  for $g\in G$ is the Dirac measure supported at $g$, and this
constitute an orthonormal basis in $L^2(G)$.
Hence
\bean
\Tr T&=&\sum_{g\in G}\langle T\d_g,\d_g\rangle=|G|\cdot\mu_1^{*(2)}(1)\\
&=&|G|\sum_{g\in G}\mu_1(g)\mu_1(g^{-1})
=|G|\|\mu_1\|_2^{2}\le 2(\dim\pi)^{2/3}.
\eean

We can also write $\Tr T=\l_1+\cdots+\l_{|G|}$ as the sum of the eigenvalues of $T$.
We can decompose the space $L^2(G)$ into the orthogonal sum of irreducible $G$ representations.
The number of components isomorphic to $\pi$  in this decomposition is $\dim \pi$.
Hence all eigenvalues of $\pi(\mu_1^{*(2)})$ occur with multiplicity at least $\dim\pi$ among
the eigenvalues of $T$.
Thus
\[
\|\pi(\mu)\|^{2^{L+1}l_1}=\|\pi(\mu_1^{*(2)})\|\le\frac{ 2(\dim\pi)^{2/3}}{\dim\pi}.
\]
Taking this inequality to the $1/2^{L+1}l_1$ power, we get the conclusion of the theorem.
\bibliography{varju}
\bibliographystyle{alpha}

\end{document}